\documentclass[a4paper,UKenglish,cleveref, thm-restate]{lipics-v2021}

\hideLIPIcs  %

\title{An FPT Algorithm for Splitting a Necklace Among Two Thieves}

\author{Michaela Borzechowski}{Department of Mathematics and Computer Science, Freie Universität Berlin, Germany}{michaela.borzechowski@fu-berlin.de}{}{DFG within the Research Training Group GRK~2434 \emph{Facets of Complexity}.}

\author{Patrick Schnider}{Department of Computer Science, ETH Zurich, Switzerland}{patrick.schnider@inf.ethz.ch}{https://orcid.org/0000-0002-2172-9285}{}

\author{Simon Weber}{Department of Computer Science, ETH Zurich, Switzerland}{simon.weber@inf.ethz.ch}{https://orcid.org/0000-0003-1901-3621}{Swiss National Science Foundation under project no. 204320.}

\authorrunning{M. Borzechowski, P. Schnider and S. Weber} %

\Copyright{Michaela Borzechowski, Patrick Schnider and Simon Weber} %



\category{} %

\relatedversion{} %

\nolinenumbers %

\EventEditors{John Q. Open and Joan R. Access}
\EventNoEds{2}
\EventLongTitle{42nd Conference on Very Important Topics (CVIT 2016)}
\EventShortTitle{CVIT 2016}
\EventAcronym{CVIT}
\EventYear{2016}
\EventDate{December 24--27, 2016}
\EventLocation{Little Whinging, United Kingdom}
\EventLogo{}
\SeriesVolume{42}
\ArticleNo{23}

\usepackage{enumitem}
\usepackage[utf8]{inputenc}
\usepackage{amsmath,amssymb,amsfonts,mathrsfs,amsthm}
\usepackage[capitalise]{cleveref}
\newcommand{\abs}[1]{\lvert #1 \rvert}

\crefname{claim}{claim}{claims}
\crefname{observation}{observation}{observations}

\usepackage{tikz}
\usepackage{xspace}

\newcommand{\erdos}{Erdős\xspace}

\newcommand{\f}[1]{\relax\ifmmode#1\else{$#1$}\fi}

\newcommand{\dimension}{\f{n}\xspace}
\newcommand{\sep}{\f{k}\xspace}
\newcommand{\Colors}{\f{C}\xspace}

\newcommand{\Reals}{\mathbb{R}\xspace}
\newcommand{\TwoThiefNecklaceSplitting}{2-Thief-Necklace-Splitting\xspace}
\newcommand{\thiefA}{\f{A^+}\xspace}
\newcommand{\thiefB}{\f{A^-}\xspace}

\newcommand{\nodeA}{\f{a}\xspace}
\newcommand{\nodeB}{\f{b}\xspace}
\newcommand{\nodeC}{\f{c}\xspace}

\newcommand{\const}{\f{k}\xspace}
\newcommand{\constB}{\f{\ell}\xspace}

\newcommand{\bound}{\omega}

\usepackage{tikz}
\usepackage{subcaption}
\usepackage[ruled]{algorithm}
\usepackage{algpseudocodex}

\algrenewcommand\algorithmicrequire{\textbf{Input:}}
\algrenewcommand\algorithmicensure{\textbf{Output:}}

\newcommand{\NP}{\ensuremath{\mathsf{NP}}}
\newcommand{\coNP}{\ensuremath{\mathsf{co}}-\ensuremath{\mathsf{NP}}}
\newcommand{\UEOPL}{\ensuremath{\mathsf{UEOPL}}}
\newcommand{\PPA}{\ensuremath{\mathsf{PPA}}}

\begin{document}

\maketitle

\begin{abstract}
It is well-known that the 2-Thief-Necklace-Splitting problem reduces to the discrete Ham Sandwich problem. In fact, this reduction was crucial in the proof of the \PPA-completeness of the Ham Sandwich problem [Filos-Ratsikas and Goldberg, STOC'19]. Recently, a variant of the Ham Sandwich problem called $\alpha$-Ham Sandwich has been studied, in which the point sets are guaranteed to be well-separated [Steiger and Zhao, DCG'10]. The complexity of this search problem remains unknown, but it is known to lie in the complexity class \UEOPL [Chiu, Choudhary and Mulzer, ICALP'20]. We define the analogue of this well-separability condition in the necklace splitting problem --- a necklace is \emph{$n$-separable}, if every subset $A$ of the $n$ types of jewels can be separated from the types $[n]\setminus A$ by at most $n$ separator points. By the reduction to the Ham Sandwich problem it follows that this version of necklace splitting has a unique solution.

We furthermore provide two FPT algorithms: The first FPT algorithm solves 2-Thief-Necklace-Splitting on $(n-1+\ell)$-separable necklaces with $n$ types of jewels and $m$ total jewels in time $2^{O(\ell\log\ell)}+m^2$. In particular, this shows that 2-Thief-Necklace-Splitting is polynomial-time solvable on $n$-separable necklaces. Thus, attempts to show hardness of $\alpha$-Ham Sandwich through reduction from the 2-Thief-Necklace-Splitting problem cannot work.
The second FPT algorithm tests $(n-1+\ell)$-separability of a given necklace with $n$ types of jewels in time $2^{O(\ell^2)}\cdot n^4$. In particular, $n$-separability can thus be tested in polynomial time, even though testing well-separation of point sets is \coNP-complete [Bergold et al., SWAT'22].
\end{abstract}

\section{Introduction}
The \emph{necklace splitting problem} is one of the most famous problems in fair division.
It is usually illustrated by the following story: two thieves have stolen a valuable necklaces with $n$ different types of jewels (diamonds, rubies, etc.).
They want to divide their bounty fairly between them, that is, in such a way that both of them get the same number of jewels of each type.
As cutting through the necklace takes a lot of effort, they want to do this with as few cuts as possible.
A mathematically inclined thief who knows the necklace splitting theorem~\cite{alonSplitting, AlonWest, GoldbergWest} will realize that no matter how the jewels are ordered on the necklace, $n$ cuts will always suffice for this.
However, all known proofs of this result are of a topological nature and do not give our thief any information on how to find the cuts.
Thus, a more algorithmically inclined thief might wonder whether a set of $n$ cuts can be found efficiently.
Unfortunately, it turns out that the search problem of finding $n$ cuts is in general \PPA-complete \cite{HamSandwichPPAComplete}, making an efficient algorithm unlikely.
In this paper, we study separability conditions under which the thieves can find the cuts efficiently.

The ideas for the separability conditions stem from a variant of another famous fair division problem, namely the \emph{Ham Sandwich problem}.
The Ham Sandwich theorem \cite{HS} states that any $d$ point sets (or mass distributions) can be simultaneously bisected by a single hyperplane.
Again, finding such a Ham Sandwich Cut is in general \PPA-complete \cite{HamSandwichPPAComplete}.
However, under the assumption that the point sets are \emph{well-separated} (which we will formally define in \cref{sec:preliminaries}), the cut is unique \cite{originalDiscreteAlphaHS} and the corresponding search problem lies in the complexity class \UEOPL \cite{aHSinUEOPL}, a subclass of \PPA, which is conjectured to be a strict subclass.

The Ham Sandwich problem and the necklace splitting problem are intimately related.
In fact, the necklace splitting theorem can be proved by lifting the necklace with $n$ types of jewels to the \emph{moment curve} in $\Reals^n$, which is the curve parameterized by $(t,t^2,t^3,\ldots,t^n)$, and then applying the Ham Sandwich theorem.
By the same idea, the \PPA-hardness for the Ham Sandwich problem follows from the \PPA-hardness of the necklace splitting problem.
In the well-separated setting, no hardness result is known for finding the now unique Ham Sandwich cut.
A natural approach to show for example \UEOPL-hardness of this problem would be to show hardness for a necklace splitting variant whose lifts give well-separated point sets.
This leads to the definition of \emph{$n$-separable necklaces}, which we again define formally in \cref{sec:preliminaries}.

However, as we show in this paper, this approach will not work, as the necklace splitting problem on $n$-separable necklaces can be solved in polynomial time.
Relaxing the notion of separability further, we get an FPT algorithm for the necklace splitting problem, parameterized by the separability:

\begin{restatable}{theorem}{polytimeNecklaceSplitting}\label{thm:polytimeNecklaceSplitting_(n-1+l)-separable}
    \TwoThiefNecklaceSplitting can be solved in time $2^{O(\constB\log\constB)}+(\sum_{c\in\Colors}|c|)^2$ on $(\dimension-1+\constB)$-separable necklaces \Colors with $\dimension$ colors.
\end{restatable}

We also provide an FPT algorithm to check whether a necklace is $(\dimension-1+\constB)$-separable.
This is again in contrast to the Ham Sandwich problem, where it has been shown that checking well-separation is \coNP-complete \cite{WellSeparationCoNP}.

Our work provides the first FPT viewpoint on the necklace splitting problem, which so far has only been studied from the viewpoint of approximation algorithms \cite{alonApproximation}.

\section{Preliminaries}\label{sec:preliminaries}
\subsection{Separability and Unique Solutions}

\begin{definition}[Necklace]
    A \emph{necklace} is a set $C$ of disjoint finite subsets of $\Reals$. The sets in~$C$ are called \emph{colors}.
\end{definition}

Note that in the literature, this kind of necklace is also called an \emph{open} necklace, since the colors are arranged in $\Reals$ and not on a cycle.

For simplicity, in the rest of this paper we assume that each color has an \emph{odd} number of points. All of our results can be adapted to the more general setting without this restriction, or even to the setting where colors are finite unions of intervals. However, the definitions and proofs have to be adjusted carefully. We discuss these possible extensions of our results in \Cref{app:non-odd}.

Since colors in a necklace are disjoint, we can view our necklace as a string over the alphabet $\Colors$: each color defines one character and the sequence of characters is defined by the
order in which the colors appear when going from $-\infty$ to $\infty$, with consecutive occurrences of the same color yielding just one character.
See \cref{fig:exampleOfSeparability} for an example.

We call the number of occurrences of a color $c$ in this string the number of \emph{components} it consists of.
We say a color $c \in \Colors$ is an \emph{interval}, if it consists of exactly one component. 
In other words, a color $c$ is an interval if its convex hull does not intersect any other color $c'\in\Colors$. In \cref{fig:2separable}, the green color $c$ is an interval, whereas the red color $a$ is not, it consists of two components.

\begin{definition}[Separability]
A necklace \Colors is \emph{\sep-separable} if for all $A \subseteq \Colors$ there exist \sep \emph{separator points} $s_1<\ldots<s_\sep\in\Reals$ that separate $A$ from $\Colors \setminus A$. More formally, if we alternatingly label the intervals $(-\infty,s_1],[s_1,s_2],\ldots,[s_{\sep-1},s_\sep],[s_\sep,\infty)$ with $A$ and $\overline{A}$, for every interval $I$ labelled $A$ we have $I\cap \bigcup_{c\in (\Colors\setminus A)}c=\emptyset$ and for every interval $I'$ labelled $\overline{A}$ we have $I'\cap \bigcup_{c\in A} c = \emptyset$.

The \emph{separability} $sep(\Colors)$ of a necklace $\Colors$ is the minimum integer $k\geq 0$ such that $\Colors$ is $k$-separable.
\end{definition}

\begin{figure}[h!]

\begin{subfigure}[b]{0.3\textwidth}
\centering
\includegraphics[page=1]{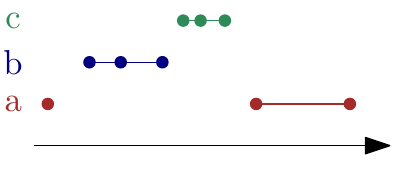}
\caption{``\nodeA \nodeB \nodeC \nodeA'' is 2-separable.}
\label{fig:2separable}
\end{subfigure}
\hfill
\begin{subfigure}[b]{0.3\textwidth}
\centering
\includegraphics[page=2]{figs/separability.pdf}
\caption{``\nodeA \nodeB \nodeA \nodeC'' is 3-separable.}
\label{fig:3separable}
\end{subfigure}
\hfill
\begin{subfigure}[b]{0.3\textwidth}
\centering
\includegraphics[page=3]{figs/separability.pdf}
\caption{``\nodeA \nodeB \nodeA \nodeB \nodeC'' is 4-separable.}
\label{fig:4separable}
\end{subfigure}
\caption{Necklace \Colors with 3 colors \nodeA, \nodeB and \nodeC.}
\label{fig:exampleOfSeparability}
\end{figure}

Note that for a necklace with $n$ colors, $sep(C)\geq n-1$, and this is tight, as can be seen in \Cref{fig:2separable}.
Our definition of \sep-separability is strongly related to the well known notion of \emph{well-separation}.

\begin{definition}
Let $P_1, \dots, P_\sep \subset \Reals^d$ be point sets. They are \emph{well-separated} if and only if for every non-empty index set $I \subset [\sep]$, the convex hulls of the two disjoint subfamilies  $\bigcup_{i \in I} P_i$ and $\bigcup_{i \in [\sep]\setminus I} P_i$ can be separated by a hyperplane.
\end{definition}

A set of two colors in $\Reals$ is $1$-separable if and only if it is well-separated.
Furthermore we observe the following property.

\begin{lemma}
\label{lem:dSeparable<=>wellSeparable}
Let \Colors be a set of \dimension colors in $\Reals$. Let $\Colors'$ be the set of subsets of $\Reals^\dimension$ obtained by lifting each point in each color of $\Colors$ to the \dimension-dimensional moment curve using the function $f(t)=(t,t^2,\ldots,t^\dimension)$.
Then the set \Colors is \dimension-separable if and only if $\Colors'$ is well-separable.
\end{lemma}

\begin{proof}

If \Colors is \dimension-separable, for each subset $A$ of \Colors, there exist \dimension points $S=(s_1, \dots, s_\dimension)$ partitioning \Colors into intervals alternatingly labelled $A$ and $\overline{A}$.
Let $H$ be the hyperplane that goes through these separator points $S$ lifted to the moment curve.
By \cite[Lemma~5.4.2]{matousek2002lectures}, at each separating point, the moment curve passes from one side of $H$ to the other.
The points belonging to intervals labelled $A$ lie on one side of the hyperplane and the points belonging to intervals labelled $\overline{A}$ lie on the other side.
Since this holds for all subsets of \Colors, it follows that $\Colors'$ is well-separated.

If $\Colors'$ is well-separated, for each subset $A'$ of colors, there exists a hyperplane that separates $A'$ from $\Colors' \setminus A'$.
By \cite[Lemma~5.4.2]{matousek2002lectures}, this hyperplane intersects the moment curve at at most \dimension points. These points define the separator points that show that \Colors is \dimension-separable.
\end{proof}

The problem of \emph{Necklace Splitting} is that two thieves want to split the necklace they stole into equal parts with as few cuts as possible.
Mathematically we partition the necklace into several intervals which belong to each thief in turn.

\begin{definition}[\TwoThiefNecklaceSplitting]\label{def:TwoThiefNecklaceSplitting}
Given a necklace \Colors with \dimension colors, find \dimension \emph{split points} that split the necklace into $\dimension + 1$ open intervals alternatingly labelled \thiefA and \thiefB, such that for each color $c\in C$, the union of all intervals labelled \thiefA contains the same number of points of $c$ as the union of all intervals labelled \thiefB.
\end{definition}

It is well known that there always exists a solution 
to this problem \cite{alonSplitting, AlonWest, GoldbergWest}.
Note that due to our assumption of every color containing an odd number of points, every solution must contain exactly one point per color as a split point.

\begin{figure}[h!]

\begin{subfigure}[b]{0.3\textwidth}
\centering
\includegraphics[page=5]{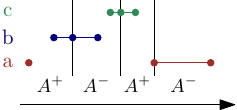}
\caption{Solution for ``\nodeA \nodeB \nodeC \nodeA''.}
\label{fig:2separableSolution}
\end{subfigure}
\hfill
\begin{subfigure}[b]{0.3\textwidth}
\centering
\includegraphics[page=6]{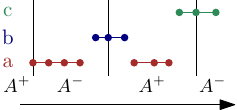}
\caption{Solution for ``\nodeA \nodeB \nodeA \nodeC''.}
\label{fig:3separableSolution}
\end{subfigure}
\hfill
\begin{subfigure}[b]{0.3\textwidth}
\centering
\includegraphics[page=7]{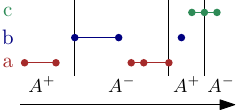}
\caption{Solution for ``\nodeA \nodeB \nodeA \nodeB \nodeC''.}
\label{fig:4separableSolution}
\end{subfigure}
\caption{Example of solutions to \TwoThiefNecklaceSplitting.}
\label{fig:NecklaceExample}
\end{figure}

\begin{theorem}
\label{thm:uniquenessOfSolution}
Let \Colors be a \dimension-separable necklace with \dimension colors.
There is a unique solution to \TwoThiefNecklaceSplitting on \Colors.
\end{theorem}

In order to prove the above theorem, we consider the classical reduction of \TwoThiefNecklaceSplitting to the Ham-Sandwich problem obtained by lifting the points to the moment curve, as it appeared in many works before \cite{momentCurve2,HamSandwichPPAComplete,matousek2002lectures,momentCurve1}. However, since the necklace we apply this reduction to is $\dimension$-separable, by \Cref{lem:dSeparable<=>wellSeparable}, the resulting points are well-separated, which allows us to apply the following stronger version of the Ham-Sandwich theorem due to Steiger and Zhao~\cite{originalDiscreteAlphaHS}.

\begin{lemma}[{$\alpha$-Ham-Sandwich Theorem, \cite{originalDiscreteAlphaHS}}]\label{lem:alphaHS}
Let $P_1, \dots, P_\dimension \subset \Reals^\dimension$ be finite well-separated point sets in \emph{weak general position}, and let $\alpha_1, \dots, \alpha_\dimension$ be positive integers with $\alpha_i \leq \abs{P_i}$, then there exists a unique $(\alpha_1, \dots, \alpha_\dimension)$-cut, i.e., a hyperplane $H$ that contains a point from each color and such that for the closed positive halfspace $H^+$ bounded by $H$ we have $\abs{H^+ \cap P_i} = \alpha_i$.
\end{lemma}

\begin{proof}[Proof of \cref{thm:uniquenessOfSolution}]
We lift all the points in \Colors to the moment curve.
The points are in general position \cite{matousek2002lectures} (and thus also in weak general position).
By \cref{lem:dSeparable<=>wellSeparable} if \Colors is \dimension-separable, then the point sets lifted to the moment curve are well-separated. 

By the $\alpha$-Ham-Sandwich theorem there exists a unique $(\lceil\frac{|c_1|}{2}\rceil,\ldots,\lceil\frac{|c_n|}{2}\rceil)$-cut that halves all colors.
This cut is a hyperplane $H$ that goes through \dimension lifted points, one point of each color.
These points define a solution $Q =(q_1, \dots, q_\dimension)$ of \TwoThiefNecklaceSplitting.

Assume that the solution $Q$ is not unique, i.e., there is another solution $Q'\neq Q$ to \Colors.
The points $Q'$ lifted to the moment curve define another hyperplane $H' \neq H$ with one point of each color, which is also a $(\lceil\frac{|c_1|}{2}\rceil,\ldots,\lceil\frac{|c_n|}{2}\rceil)$-cut.
But by \Cref{lem:alphaHS} there is a unique hyperplane with this property, so $Q'$ cannot exist.
\end{proof}

In this proof, we do not use the property that \Cref{lem:alphaHS} guarantees that there is a solution or \emph{every} choice of $\alpha$, we merely use it for the guaranteed uniqueness of a solution for a halving cut.

Note that the opposite direction of \cref{thm:uniquenessOfSolution} does not hold, i.e., there are necklaces with \dimension colors which are not \dimension-separable but still have unique solutions for \TwoThiefNecklaceSplitting, see \cref{fig:4separableSolution} for an example.

\subsection{Graph-Theoretic Aspects}

To argue about the separability of necklaces, we wish to think about graphs rather than strings or even point sets. For every necklace, we thus define its walk graph:
\begin{definition}[Walk graph]
Given a necklace \Colors, the walk graph $G_\Colors$ is the multigraph with $V=\Colors$ and with every potential edge $\{a,b\}\in \binom{V}{2}$ having multiplicity equal to the number of substrings ``$ab$'' plus the number of substrings ``$ba$'' in the string describing \Colors.
\end{definition}

The walk graphs of the example necklaces in \cref{fig:exampleOfSeparability} can be seen in \cref{fig:WalkGraphExample}.

\begin{figure}[h!]
\begin{subfigure}[b]{0.3\textwidth}
\centering
\begin{tikzpicture}
\node (a) at (0,0) {\nodeA};
\node (b) at (2,0) {\nodeB};
\node (c) at (1,1.5) {\nodeC};
\draw (a) edge (b);
\draw (b) edge (c);
\draw (c) edge (a);
\end{tikzpicture}
\caption{Walk graph for ``\nodeA \nodeB \nodeC \nodeA''}
\end{subfigure}
\begin{subfigure}[b]{0.3\textwidth}
\centering
\begin{tikzpicture}
\node (a) at (0,0) {\nodeA};
\node (b) at (2,0) {\nodeB};
\node (c) at (1,1.5) {\nodeC};
\draw (a) edge (b);
\draw (a) edge[bend left] (b);
\draw (c) edge (a);
\end{tikzpicture}
\caption{Walk graph for ``\nodeA \nodeB \nodeA \nodeC''}
\end{subfigure}
\begin{subfigure}[b]{0.3\textwidth}
\centering
\begin{tikzpicture}
\node (a) at (0,0) {\nodeA};
\node (b) at (2,0) {\nodeB};
\node (c) at (1,1.5) {\nodeC};
\draw (a) edge (b);
\draw (a) edge[bend left] (b);
\draw (a) edge[bend right] (b);
\draw (c) edge (b);
\end{tikzpicture}
\caption{Walk graph for ``\nodeA \nodeB \nodeA \nodeB \nodeC''}
\end{subfigure}
\caption{Walk graphs of the examples in \cref{fig:exampleOfSeparability}.}
\label{fig:WalkGraphExample}
\end{figure}
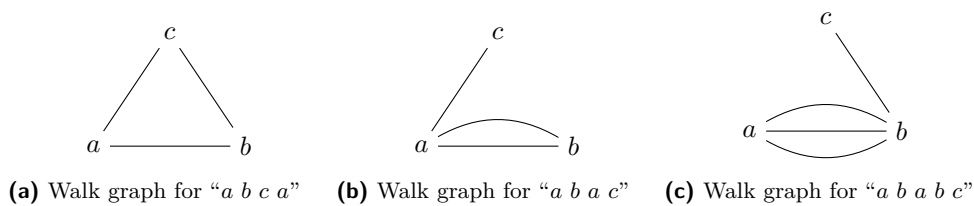

Note that given a necklace \Colors as a set of point sets, both the string describing it as well as the walk graph can be built in linear time in the size of the necklace $\sum_{c\in\Colors}|c|$ .

\begin{observation}\label{obs:semiEulerian}
The walk graph of a necklace is connected and semi-Eulerian.
\end{observation}

Recall the following well-known fact about semi-Eulerian graphs.
\begin{lemma}\label{lem:odddegrees}
    In a semi-Eulerian (multi-)graph, at most two vertices have odd degrees.
\end{lemma}

The separability of a necklace turns out to be equivalent to the max-cut in its walk graph.

\begin{definition}[Cut]
    In a (multi-)graph $G$ on the vertices $V$, a \emph{cut} is a subset $A\subseteq V$. The \emph{size of a cut} $A$ is the number of edges $\{u,v\}$ in $G$ such that $u\in A$ and $v\not\in A$. The \emph{max-cut}, denoted by $\mu(G)$, is the largest size of any cut $A\subseteq V$. 
\end{definition}
\begin{lemma}
    For every necklace $\Colors$, we have $sep(\Colors)=\mu(G_\Colors)$.
\end{lemma}
\begin{proof}
For every subset $A\subseteq \Colors$, the number of separator points needed to separate the colors in $A$ from $\Colors\setminus A$ is given by the size of the cut $A$ in $G_C$, since the edges going over this cut correspond one to one to the points in the necklace where the necklace switches from a color in $A$ to a color not in $A$, or vice versa. Thus, the max-cut $\mu(G_\Colors)$ corresponds to the maximal number of separator points we need to separate any two subsets of colors.
\end{proof}

In our proofs we will often show that certain structures or properties do not appear in walk graphs of necklaces with bounded separability. The general strategy for these proofs will be to show that walk graphs with these structures or properties have a large max-cut, and thus the corresponding necklaces cannot have the claimed separability. Our main tool for this is the following bound, originally conjectured by \erdos~\cite{erdosConjecture} and proven by Edwards~\cite{edwards_1,edwards_2}.

\begin{theorem}[Edwards-\erdos bound]
    A simple connected graph $G$ with $n$ vertices and $m$ edges has a maximum cut $\mu(G)$ of at least $\bound(G):=\frac{m}{2}+\frac{n-1}{4}$.
\end{theorem}

Since walk graphs are not simple graphs, we will use a corollary of the following strengthening, due to Poljak and Turzík~\cite{poljakBoundForNonsimple}:
\begin{theorem}[\cite{poljakBoundForNonsimple}]
    For a connected graph $G$ with weight function $w:E\rightarrow\Reals_+$, there exists a cut of weight at least
    \[\frac{\sum_{e\in E}w(e)}{2}+\frac{t(G,w)}{4},\]
    where $t(G,w)$ is the weight of a minimum-weight spanning tree of $G$.
\end{theorem}
\begin{corollary}\label{cor:multigraphErdos}
A connected (multi-)graph $G$ with $n$ vertices and $m$ edges has a maximum cut $\mu(G)$ of at least $\bound(G):=\frac{m}{2}+\frac{n-1}{4}$.
\end{corollary}

For determining the separability of a necklace, we will use an algorithm due to Crowston, Jones and Mnich~\cite{blackboxFPT} to decide max-cut beyond the Edwards-\erdos bound.

\begin{theorem}[FPT algorithm \cite{blackboxFPT}]
\label{thm:blackbox}
There exists an algorithm that decides whether for a given simple connected graph $G$ with \dimension vertices and $m$ edges the max-cut $\mu(G)$ is at most
$\bound(G)+ \const$
in time $2^{O(\const)}\cdot \dimension^4$.
\end{theorem}

This is a so-called fixed-parameter algorithm; for any \emph{fixed} parameter $k$, the algorithm runs in polynomial time (in $n$). Note again that this algorithm only works on simple graphs, thus, we will need to alter the walk graphs to be able to apply this algorithm.

\section{An FPT Algorithm for \TwoThiefNecklaceSplitting}

In this section we show \Cref{thm:polytimeNecklaceSplitting_(n-1+l)-separable}:

\polytimeNecklaceSplitting*

The algorithm we use is recursive, based on the following crucial observation.

\begin{theorem}
\label{thm:intervalOrTwoComponents}
Let \Colors be an $(\dimension -1 + \constB)$-separable necklace with \dimension colors. If $\dimension\geq 6 \constB +2$ there must exist
\begin{itemize}
    \item[(i)] two neighboring colors that are both intervals, or
    \item[(ii)] one color that only consists of exactly two components.
\end{itemize}
\end{theorem}

\begin{proof}
Since the walk graph is semi-Eulerian, it contains either $0$ or $2$ vertices with odd degree (recall \Cref{obs:semiEulerian,lem:odddegrees}). A color that is an interval has degree $2$, unless it is at the beginning or end of the necklace. A color that consists of more than two components has degree at least $6$ (or $5$ or $4$ if it is at the beginning and/or end of the necklace).

Let $A\subseteq C$ be the set of intervals. Note that if no two intervals are neighboring, we can pick all the intervals as a cut $A$, which has size at least $\mu(A) \geq 2\abs{A}-2$. Since we know that $\mu(G_C)\leq \dimension-1+\constB$, we must have that $\abs{A}\leq \frac{\dimension +1 +\constB}{2}$.

Assume that the theorem does not hold, and that there thus exist no neighboring intervals and no color consisting of exactly two components. We can then bound the sum of degrees 
$\sum_{c\in \Colors}deg(c)
\geq 2\cdot \frac{\dimension +1 +\constB}{2}
+ 6\cdot (\dimension - \frac{\dimension +1 +\constB}{2}) - 2 
= 4\dimension - 2 \constB - 4$. 
Thus, the number of edges $\abs{E}$ in $G_C$ is bounded $\abs{E}\geq \frac{ 4\dimension - 2 \constB - 4}{2}=2\dimension -\constB -2$. 
Due to \Cref{cor:multigraphErdos} we thus get that 
$\mu(G_C)\geq \frac{2 \dimension -\constB - 2}{2}+\frac{\dimension-1}{4}=\frac{5}{4}\dimension- \frac{\constB}{2} - \frac{5}{4}$. 
By the assumption $\dimension\geq 6\constB+2$, we therefore have $\mu(G_C)\geq \dimension - \frac{3}{4} + \constB$, which is a contradiction to the assumption that $\mu(G_C)\leq \dimension-1+\constB$. Thus, the theorem follows.
\end{proof}

To use  \cref{thm:intervalOrTwoComponents} to recursively solve smaller instances, we need to make sure that the separability of the smaller instances translates back to the separability of the original instance. The following two lemmas provide this necessary correspondence.

\begin{lemma}
\label{lem:SeparabilityWhenRemovingInterval}
Let \Colors be a necklace. Let $\Colors'$ be the necklace obtained by removing two neighboring intervals $c,c'$ from \Colors. Then, $sep(\Colors')= sep(\Colors)-2$.
\end{lemma}
\begin{proof}
    In the walk graph, removing two neighboring intervals corresponds to replacing a path $(a,c,c',b)$ of length $3$ by a direct edge connecting $a$ and $b$.
    
    Every cut $A'\subseteq \Colors'$ in $G_{\Colors'}$ of size $k$ can be extended to a cut $A\subseteq\Colors$ in $G_\Colors$ of size $k+2$: For every vertex $v\in \Colors'$, we have $v\in A'$ iff $v\in A$. Furthermore, $c\in A$ iff $a\not\in A'$ and $c'\in A$ iff $b\not\in A'$. Thus, $sep(C)\geq sep(C')+2$. 

    Similarly, every cut $A\subseteq C$ in $G_C$ of size $k$ induces a cut $A'=A\cap C'$ of size $k-2$ in $G_{C'}$. Thus, $sep(C')\geq sep(C)-2$, and we get $sep(C')=sep(C)-2$.
\end{proof}

\begin{lemma}
\label{lem:SeparabilityWhenCollapsingTwoComponents}
Let \Colors be a necklace on \dimension colors that is $(\dimension-1+\constB)$-separable. The necklace $\Colors'$ obtained by reducing a color $c\in\Colors$ to a subset $\emptyset\subset c'\subset c$ is still $(\dimension-1+\constB)$-separable. 
\end{lemma}
\begin{proof}
    By simplifying a necklace, we can not increase its separability.
\end{proof}

We are now ready to present \cref{alg:solveNecklaceSplitting}, an FTP algorithm to solve \TwoThiefNecklaceSplitting on $(\dimension-1+\constB)$-separable necklaces.
The strategy is to reduce the given necklace either by removing two neighboring intervals, or by removing one of the two components in a color that consists of exactly two components. By \Cref{lem:SeparabilityWhenRemovingInterval,lem:SeparabilityWhenCollapsingTwoComponents}, if \Colors is $(\dimension -1 + \constB)$-separable, the resulting necklace $\Colors'$ is again $(\dimension' -1 + \constB)$-separable (for $n'=|C'|$), and can thus be solved recursively. The solution of the reduced case is then extended back to a solution of the original necklace. A necklace can be reduced as long as \Cref{thm:intervalOrTwoComponents} applies, and thus we only need to solve the case $\dimension<6\constB + 2$ directly.

For an example of the execution of the algorithm, see \Cref{fig:algorun1} and \Cref{fig:algorun2}. Note that these small instances would technically be solved by brute-force and merely serve as illustrations.

\begin{algorithm}[h!]
\caption{\textsc{RecursiveNS}}\label{alg:solveNecklaceSplitting}
\begin{algorithmic}[1]
\Statex{\textbf{Input:} An $(\dimension -1 + \constB)$-separable necklace \Colors with \dimension colors.}
\Statex{\textbf{Output:} \dimension split points.}

\If{$\dimension < 6\constB + 2$}
    \State{$Q\gets $ \textsc{BruteForce}($C$)}\label{alg:solveNecklaceSplitting:bruteforce}
    \State \Return{$Q$}
\ElsIf{there exist two neighboring intervals $c,c'\in C$}
    \State{$Q\gets$ \textsc{RecursiveNS}($C\setminus\{c,c'\}$)}
    \State \Return{$Q\cup\{median(c),median(c')\}$}\label{alg:solveNecklaceSplitting:return1}
\Else
    \State{$c\gets$ a color consisting of two components $c_1,c_2$}\label{alg:solveNecklaceSplitting:findTwoComponent}
    \State{$c'\gets $ largest component of $c$}

    \If{$|c'|$ is even} \label{alg:solveNecklaceSplitting:median1}
        \State{Add a median point to $c'$} \label{alg:solveNecklaceSplitting:median2}
    \EndIf
    \State{$Q\gets$ \textsc{RecursiveNS}($(C\setminus \{c\})\cup \{c'\}$)}
    \State{$\{q\}\gets Q\cap c'$}
    \State{$q'\gets q$ shifted right/left by $\lceil\frac{min(|c_1|,|c_2|)}{2}\rceil$ points of $c'$}\label{alg:solveNecklaceSplitting:adjustSplit} \Comment{direction depending on parity of number of split points in $Q$ between $c_1$ and $c_2$}
    \State\Return{$Q\setminus\{q\}\cup\{q'\}$}\label{alg:solveNecklaceSplitting:return2}
\EndIf

\end{algorithmic}
\end{algorithm}

\begin{proof}[Proof of \Cref{thm:polytimeNecklaceSplitting_(n-1+l)-separable}]
We first argue for correctness of \Cref{alg:solveNecklaceSplitting}. 
By \Cref{thm:intervalOrTwoComponents}, if we reach line \ref{alg:solveNecklaceSplitting:findTwoComponent} we can always find a color which consist out of exactly two components, so the algorithm can never fail to finish.

We have to argue that our algorithm returns a correct solution in both line \ref{alg:solveNecklaceSplitting:return1} and line \ref{alg:solveNecklaceSplitting:return2}.
\begin{itemize}
\item[(i)] Line \ref{alg:solveNecklaceSplitting:return1}: The constructed solution splits the two neighboring intervals correctly. Since we place two splits, the parity of the partition outside of these intervals does not change in comparison to the solution $Q$ obtained recursively. Thus, all other colors are also split correctly.

\item[(ii)] Line \ref{alg:solveNecklaceSplitting:return2}: The constructed solution splits color $c$ correctly, and $q'$ lies in the same component of $c$ as $q$, since $c'$ is the larger of the two components. Shifting the split within the same component of $c$ does not change the partition outside of this component in comparison to the solution $Q$ obtained recursively. Thus, all other colors are also split correctly.
\end{itemize}

It remains to argue for the runtime of \Cref{alg:solveNecklaceSplitting}. Clearly, we only use the brute-force approach at line \ref{alg:solveNecklaceSplitting:bruteforce} once. In an $(\dimension-1+\constB)$-separable necklace with $\dimension<6\constB+2$, each color has at most $O(\constB)$ components. For each guess of one component per color, it can be determined in polynomial time in $\constB$ whether this guess admits a solution. There are at most $\constB^{O(\constB)}$ guesses, thus we can solve this base case in time $2^{O(\constB\log\constB)}$.

In the rest of the algorithm, on each level of the recursion we reduce the number of points in the necklace by at least one, and we can make the necessary adjustments and find the needed colors in linear time in the number of points. Thus, the total runtime of the algorithm is $2^{O(\constB\log\constB)}+(\sum_{c\in\Colors}|c|)^2$, as claimed.
\end{proof}

\begin{figure}[h!]
    \begin{subfigure}{0.3\textwidth}
        \centering
        \includegraphics[page=1]{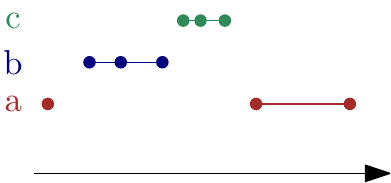}
        \caption{Original necklace.}
    \end{subfigure}
    \hfill
    \begin{subfigure}{0.3\textwidth}
        \centering
        \includegraphics[page=2]{figs/algorun1.pdf}
        \caption{Reduced necklace.}
    \end{subfigure}
    \hfill
    \begin{subfigure}{0.3\textwidth}
        \centering
        \includegraphics[page=3]{figs/algorun1.pdf}
        \caption{Solution.}
    \end{subfigure}
    \caption{Example step of \Cref{alg:solveNecklaceSplitting} using the reduction of removing two neighboring intervals (\nodeB and \nodeC).}\label{fig:algorun1}
\end{figure}

\begin{figure}[h!]
    \begin{subfigure}{0.3\textwidth}
        \centering
        \includegraphics[page=1]{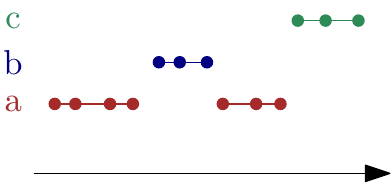}
        \caption{Original necklace.}
    \end{subfigure}
    \hfill
    \begin{subfigure}{0.3\textwidth}
        \centering
        \includegraphics[page=2]{figs/algorun2.pdf}
        \caption{Reduced necklace.}
    \end{subfigure}
    \hfill
    \begin{subfigure}{0.3\textwidth}
        \centering
        \includegraphics[page=3]{figs/algorun2.pdf}
        \caption{Solution.}
    \end{subfigure}
    \caption{Example step of \Cref{alg:solveNecklaceSplitting} using the reduction of removing a component from the two-component color \nodeA.}\label{fig:algorun2}
\end{figure}

For the special case of \dimension-separable necklaces, i.e., $\constB = 1$, we get the following corollary:
\begin{corollary}
\label{cor:polytimeNecklaceSplitting_n-separable}
Finding the unique solution for \TwoThiefNecklaceSplitting on an $\dimension$-separable necklace with \dimension colors takes polynomial time.
\end{corollary}

Until now, both \cref{thm:polytimeNecklaceSplitting_(n-1+l)-separable} and \cref{cor:polytimeNecklaceSplitting_n-separable} work under the initial promise that \Colors is $(\dimension-1+\constB)$-separable (or \dimension-separable respectively). If the algorithm fails because none of the cases applies, this certifies that the input necklace was \emph{not} $(\dimension-1+\constB)$-separable. On the other hand, \Cref{alg:solveNecklaceSplitting} may run successfully, even if the input necklace is not $(\dimension-1+\constB)$-separable, and if it does run successfully, its output is always a correct solution. Since \Cref{alg:solveNecklaceSplitting} can produce these ``false positives'', it cannot be used to decide $(\dimension-1+\constB)$-separability. We tackle that problem in the next section.

\section{Testing Separability}

At first, it seems like finding a polynomial-time algorithm for deciding whether a necklace is $(\dimension-1+\constB)$-separable may be futile, since we have the following theorem due to Guruswami~\cite{maxCutEulerian}:

\begin{theorem}[\cite{maxCutEulerian}]
Given a Eulerian graph $G$ and an integer $\sep$, deciding whether the size of the max-cut $\mu(G)\geq \sep$ is \NP-complete.
\end{theorem}
Since to compute the separability of a necklace we need to compute the max-cut of its walk graph, and since every Eulerian graph is the walk graph of some necklace\footnote{Simply find a Eulerian path through the graph and place one point per character in the respective color. If some color has an even number of points, add one more to an existing component.}, we get the following corollary:
\begin{corollary}
\label{cor:k-SeparabilityIsNPComplete}
    Given a necklace \Colors of \dimension colors and an integer \sep, deciding whether \Colors is \sep-separable is \coNP-complete.
\end{corollary}

However, not all hope is lost. To check whether a necklace is  $(\dimension-1+\constB)$-separable, we do not need to compute the max-cut of its walk graph, we merely need to check whether it is at most $(\dimension-1+\constB)$. We next provide an FPT algorithm that checks $(\dimension - 1 + \constB)$-separability for fixed parameter \constB. With $\constB=1$ this shows that testing \dimension-separability of \dimension colors is solvable in polynomial time, even though both testing \sep-separability of \dimension colors with \sep as input as well as testing well-separation of point sets are \coNP-complete~\cite{WellSeparationCoNP}. More generally, we show the following theorem:
\begin{theorem}\label{thm:ourFPT}
    There exists an FPT algorithm for fixed parameter $\constB$ that can decide whether the max-cut of a given semi-Eulerian multigraph $G_C$ with \dimension vertices is at most $\dimension-1+\constB$, i.e., it can decide whether $\mu(G_C)\leq \dimension-1+\constB$ in time $2^{O(\constB^2)}\cdot \dimension^4$.
\end{theorem}

By \cref{thm:blackbox}, there exists an algorithm that decides whether a \emph{simple} graph $G$ with \dimension vertices and a fixed paramter \const has a max-cut of size $\mu(G) \geq \bound(G)+\const = \frac{\abs{E(G)}}{2}+ \frac{\dimension-1}{4}+\const$ in $2^{O(\const)}\cdot \dimension^4$ time.
But our input graph $G_\Colors$ is a multigraph and we have no bound on its number of edges, nor on the distance between $\bound(G_\Colors)$ and $\dimension-1+\constB$.
In order to use this algorithm to decide separability, we need the following:
\begin{enumerate}

\item Derive a graph $G_\Colors'$ from $G_\Colors$ such that we \emph{can} bound $\abs{E(G_\Colors')}$ and thus $\bound(G_\Colors')$.

\item Prove that there is a bounded number of multi-edges in $G_\Colors'$.
\item Transform $G_\Colors'$ into a simple graph $G_\Colors''$ by blowing up its multi-edges by a constant factor.
\end{enumerate}

An interval in \Colors is equivalent to a vertex in $G_\Colors$ with degree at most $2$.
In the following, we will use the term \emph{interval} for such vertices and intervals on necklaces interchangeably.

\begin{lemma}
\label{lem:G'}
Given a semi-Eulerian multigraph $G$ on $\dimension$ vertices, we can either detect that $\mu(G)>\dimension-1+\constB$, or we can build a multigraph $G'$ on $\dimension'$ vertices such that $|E(G')|\geq \frac{3}{2}\dimension'-\frac{\constB}{2}-1$, and such that $\mu(G)\leq \dimension-1+\constB$ if and only if $\mu(G')\leq \dimension'-1+\constB$.
\end{lemma}

\begin{proof}
Given a multigraph $G$,
let $G'$ be the result of applying \Cref{lem:SeparabilityWhenRemovingInterval} on $G$ exhaustively.
As long as there are two adjacent intervals in $G$, we can remove the two intervals, thus reducing the maximum cut size by $2$.
In each such step we remove $2$ vertices, $3$ edges and add $1$ new edge.
Thanks to \Cref{lem:SeparabilityWhenRemovingInterval}, we have the desired correspondence between $\mu(G)$ and $\mu(G')$.

Assume there are at least $\frac{\dimension' +\constB}{2}$ intervals in $G'$.
Then the cut $A$ in $G'$ with all intervals on one side and all other vertices on the other side has size $\mu(A) \geq 2 \cdot \frac{\dimension' +\constB}{2} = \dimension'+\constB$. 
It follows that $\mu(G') \geq \mu(A) > \dimension' - 1 + \constB$.
In this case we can thus detect that $\mu(G)> \dimension-1+\constB$.

In the other case, there are less than $\frac{\dimension' +\constB}{2}$ intervals in $G'$.
All other vertices have degree at least $4$ (excluding the start and end vertex).
Therefore the sum of degrees in $G'$ is
\[\sum_{v \in V(G')} deg(v) \geq \frac{\dimension' +\constB}{2} \cdot 2 + \frac{\dimension' -\constB}{2} \cdot 4 -2 
=3 \dimension' -\constB-2.\]
Thus the number of edges in $G'$ is $\abs{E(G')} \geq \frac{3}{2}\dimension' -\frac{\constB}{2} -1$.
\end{proof}

We can now see the following.
\begin{observation}
\label{obs:boundedErdosBound}
Given this bound on $|E(G')|$, the bound $\bound(G')$ given by \Cref{cor:multigraphErdos} can be bounded by 
\[\bound(G')\geq\frac{\frac{3}{2}\dimension'-\frac{\constB}{2}-2}{2}+\frac{\dimension'-1}{4}=\dimension'-\frac{\constB}{4}-\frac{5}{4}.\]
\end{observation}
Thus, by the process of eliminating neighboring intervals, we have managed to get the difference between $(\dimension'-1+\constB)$ and $\bound(G_\Colors')$ to be a constant depending only on $\constB$.

Next we show that the total multiplicity $M$ of the multi-edges in $G_\Colors'$ cannot be too large. We show that if $G_\Colors'$ has maximum cut size at most $\dimension'-1+\constB$, the total multiplicity of multi-edges can be bounded by a function solely depending on $\constB$, and not $\dimension$ or $|E(G_\Colors')|$.

\begin{lemma}
\label{lem:d+l-separable=>AtMostk2Multiedges}
In a multigraph $G$ on $\dimension$ vertices with $\mu(G)\leq \dimension-1+\constB$, the total multiplicity of the multi-edges in $G$ is at most $2\constB^2$.
\end{lemma}
\begin{proof}
Let  $G'$ be a weighted simple graph with an edge of weight $m-1$ for every multi-edge of multiplicity $m\geq 2$ in the graph $G$.
Note that the total weight of $G'$ is at least half of the total multiplicity of multi-edges in $G$. 

Let $F$ be a spanning forest in $G'$ with total weight $w$. 
Given $F$, we can build a spanning tree $T$ of $G$ of total weight $n-1+w$, since every edge of $F$ of weight $m'-1$ corresponds to a multi-edge of multiplicity $m'$ in $G$, and all additional edges used to make $F$ into a spanning tree have weight $1$. Since every tree is bipartite, the weight of $T$ is a lower bound on the max-cut of $G$: $\mu(G) \geq \dimension-1+w$.
Thus, for a given $G$ with $\mu(G) \leq \dimension-1+\constB$, the total weight of $F$ must be at most $\constB$.

We thus only need to show that in a simple weighted graph (in our case, $G'$), in which every weight is at least $1$, and whose maximum-weight spanning forest has weight at most $\constB$, the total weight of the graph is at most $\constB^2$. To see this, we successively remove spanning forests from $G'$ until $G'$ is empty. Every spanning forest we remove has weight at most $\constB$. As every edge has weight at least $1$, every vertex in $G'$ has degree at most $\constB$. Thus, we are done after removing at most $\constB$ spanning forests. Thus, the total weight of $G'$ is at most $\constB^2$.

We conclude that the total multiplicity of multi-edges in $G$ can be at most $2\constB^2$.
\end{proof}

Finally, we show how $G_C'$ can be transformed into a simple graph $G_C''$. Let \nodeA and \nodeB be vertices in $G_C'$ with a multi-edge of multiplicity $m$ between them.
We construct the graph $G_C''$ from $G_C'$ by removing the multi-edge between \nodeA and \nodeB and introducing $m$ paths of length three from \nodeA to \nodeB, all going through separate vertices.
See \cref{fig:blowUpMultiedges} for an example application of this process, and \cref{alg:blowUp} for pseudo-code describing it.

\begin{algorithm}[h!]
\caption{Simplifying a multi-edge of multiplicity $m$}\label{alg:blowUp}
\begin{algorithmic}[1]
\Statex{\textbf{Input:} A Graph $G$ with \dimension vertices, two vertices \nodeA and \nodeB with a multi-edge of multiplicity $m$ between them.}
\Statex{\textbf{Output:} A Graph $G'$ with $\dimension+2m$ vertices}

\State $V(G') \gets V(G) \cup \{i_{j,1}, i_{j,2} \;\vert\; j\in [m]\}$, where $i_{j,1},i_{j,2}\notin V(G)$

\State $E(G') \gets E(G) \setminus \{\{\nodeA, \nodeB\}\}$ 
\State $E(G') \gets E(G') \cup \{\{\nodeA, i_{j,1}\},\{i_{j,1}, i_{j,2}\},\{i_{j,2}, \nodeB\} \;\vert\; j\in [m]\}$
\State \Return $G'$

\end{algorithmic}
\end{algorithm}

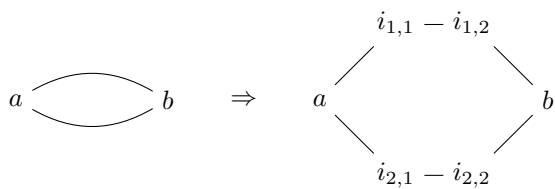
\begin{figure}[h!]
\centering
\begin{tikzpicture}
\newcommand{\offset}{4}
\node (a) at (0,0) {\nodeA};
\node (b) at (2,0) {\nodeB};
\draw (a) edge[bend left] (b);
\draw (a) edge[bend right] (b);

\node (arrow) at (3,0) {$\Rightarrow$};

\node (a2) at (0+ \offset,0) {\nodeA};
\node (b2) at (3+ \offset,0) {\nodeB};
\node (i1) at (1+ \offset,1) {$i_{1,1}$};
\node (i2) at (2+ \offset,1) {$i_{1,2}$};
\node (i3) at (1+ \offset,-1) {$i_{2,1}$};
\node (i4) at (2+ \offset,-1) {$i_{2,2}$};
\draw (a2) edge (i1);
\draw (i1) edge (i2);
\draw (i2) edge (b2);
\draw (a2) edge (i3);
\draw (i3) edge (i4);
\draw (i4) edge (b2);
\end{tikzpicture}
\caption{Example of \cref{alg:blowUp} to blow up a multi-edge of multiplicity 2 to make the graph simple.}
\label{fig:blowUpMultiedges}
\end{figure}

This process is again constructed in such a way that the change of the max-cut is predictable:
\begin{lemma}
\label{lem:blowUp}
Let $G$ be a multigraph on \dimension vertices.
Let \nodeA and \nodeB be vertices in $G$ with a multi-edge of multiplicity $m$ between them.
Let $G'$ be the result of \cref{alg:blowUp} called with $(G, \nodeA, \nodeB)$. Then, $\mu(G')=\mu(G)+2m$.
\end{lemma}

\begin{proof}
Let $A \subseteq V(G)$ be some max-cut in $G$ with $\mu(A) = \mu(G).$

We distinguish between two cases.
If the multi-edge goes across the cut, i.e. $\nodeA \in A$ and $\nodeB \notin A$,
the same cut in $G'$ has $m$ fewer edges (namely the multi-edge) and $3m$ edges more, namely all of the newly introduced edges of the paths, see \cref{fig:multiEdgeIsInMaxCut}.
If the multi-edge between \nodeA and \nodeB is \emph{not} in the max-cut of $G$, there is a cut in $G'$ that has $2m$ new edges, namely one of each newly introduced path, see \cref{fig:multiEdgeIsNotInMaxCut}.
Thus, a max-cut of size $\mu(G)$ in $G$ implies a cut of size $\mu(G) + 2m$ in $G'$, and thus $\mu(G')\geq \mu(G)+2m$. 

For the other direction, consider a max-cut $A'$ of $G'$. Since $A'$ is maximal, it must either contain all $3m$ intermediate edges between \nodeA and \nodeB, and put \nodeA and \nodeB on different sides of the cut, or it must put \nodeA and \nodeB on the same side of the cut, and contain exactly $2m$ intermediate edges (see again \Cref{fig:blowUpMultiedgesWithCut}). Thus, there must exist a cut $A$ in $G$ which contains exactly $2m$ fewer edges than $A'$, and we get $\mu(G)\geq \mu(G')-2m$.

We conclude that $\mu(G')=\mu(G)+2m$.
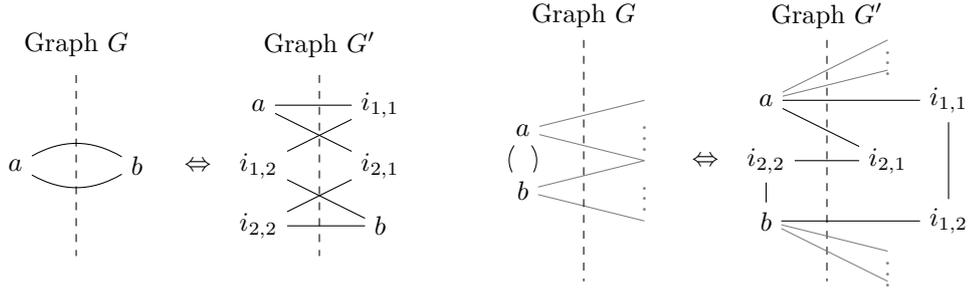
\begin{figure}[h!]
\begin{subfigure}[b]{0.5\textwidth}
\centering
\begin{tikzpicture}[scale=0.8]
\newcommand{\offset}{4}
\node (a) at (0,0) {\nodeA};
\node (b) at (2,0) {\nodeB};
\draw (a) edge[bend left] (b);
\draw (a) edge[bend right] (b);
\draw[dashed] (1,1.5) -- (1,-1.5);

\node (arrow) at (3,0) {$\Leftrightarrow$};

\node (a2) at (0+ \offset,1) {\nodeA};
\node (b2) at (2+ \offset,-1) {\nodeB};
\node (i1) at (2+ \offset,1) {$i_{1,1}$};
\node (i2) at (0+ \offset,0) {$i_{1,2}$};
\node (i3) at (2+ \offset,0) {$i_{2,1}$};
\node (i4) at (0+ \offset,-1) {$i_{2,2}$};
\draw (a2) edge (i1);
\draw (i1) edge (i2);
\draw (i2) edge (b2);
\draw (a2) edge (i3);
\draw (i3) edge (i4);
\draw (i4) edge (b2);
\draw[dashed] (1+\offset,1.5) -- (1+\offset,-1.5);
\node at (1, 2) {Graph $G$};
\node at (1+\offset, 2) {Graph $G'$};
\node at (0, -2){};
\node at (0, 2){};
\end{tikzpicture}
\caption{Case 1: The multi-edge is in max-cut of $G$.}
\label{fig:multiEdgeIsInMaxCut}
\end{subfigure}
\begin{subfigure}[b]{0.5\textwidth}
\centering
\begin{tikzpicture}[scale=0.8]
\newcommand{\offset}{4}
\node (a) at (0,0.5) {\nodeA};
\node (b) at (0,-0.5) {\nodeB};
\draw[gray] (a) edge (2,1);
\draw[gray] (a) edge (2,0);
\draw[gray] (b) edge (2,-1);
\draw[gray] (b) edge (2,0);
\node[gray] at (2, 0.5) {\vdots};
\node[gray] at (2, -0.5) {\vdots};
\draw (a) edge[bend left] (b);
\draw (a) edge[bend right] (b);
\draw[dashed] (1,2) -- (1,-2);

\node (arrow) at (3,0) {$\Leftrightarrow$};

\node (a2) at (0+ \offset,1) {\nodeA};
\node (b2) at (0+ \offset,-1) {\nodeB};
\node (i1) at (3+ \offset,1) {$i_{1,1}$};
\node (i2) at (3+ \offset,-1) {$i_{1,2}$};
\node (i3) at (2+ \offset,0) {$i_{2,1}$};
\node (i4) at (0+ \offset,0) {$i_{2,2}$};
\draw (a2) edge (i1);
\draw (i1) edge (i2);
\draw (i2) edge (b2);
\draw (a2) edge (i3);
\draw (i3) edge (i4);
\draw (i4) edge (b2);
\draw[dashed] (1+\offset,2) -- (1+\offset,-2);

\draw[gray] (a2) edge (2+\offset,2);
\draw[gray] (a2) edge (2+\offset,1.5);
\draw[gray] (b2) edge (2+\offset,-2);
\draw[gray] (b2) edge (2+\offset,-1.5);
\node[gray] at (2+\offset, 1.75) {\vdots};
\node[gray] at (2+\offset, -1.75) {\vdots};
\node at (1, 2.4) {Graph $G$};
\node at (1+\offset, 2.4) {Graph $G'$};
\end{tikzpicture}
\caption{Case 2: The multi-edge is \emph{not} in max-cut of $G$.}
\label{fig:multiEdgeIsNotInMaxCut}
\end{subfigure}
\caption{Change of max-cut size when blowing up a multi-edge of multiplicity 2.}
\label{fig:blowUpMultiedgesWithCut}
\end{figure}
\end{proof}

We are now ready to put this all together and describe
the algorithm proving \Cref{thm:ourFPT}.

\begin{algorithm}[h!]
\caption{FPT algorithm for testing $\mu(G_C)\leq \dimension-1+\constB$ with fixed parameter \constB}\label{alg:FPT}
\begin{algorithmic}[1]
\Statex{\textbf{Input:} A semi-Eulerian multigraph $G_\Colors$ on $\dimension$ vertices.}
\Statex{\textbf{Output:} True iff $\mu(G_\Colors)\leq \dimension - 1 + \constB$.}

\State $G_\Colors' \gets G_\Colors$
\While{there exist neighboring intervals in $G_\Colors'$}\label{alg:FPT:RemoveIntervals}
    \State Remove two neighboring intervals from $G_\Colors'$.
\EndWhile

\State $\dimension' \gets |V(G_\Colors')|$

\State $i \gets $ The number of intervals in $G_\Colors'$.\label{alg:FPT:CountIntervals}
\If{$i> \frac{\dimension' + \constB}{2}$}\label{alg:FPT:TooManyIntervals}
    \Return false \Comment{based on \Cref{lem:G'}}
\EndIf

\State $M \gets$ The total multiplicity of multi-edges in $G_\Colors'$.

\If{$M > 2\constB^2$}
    \Return false \Comment{based on \Cref{lem:d+l-separable=>AtMostk2Multiedges}}
\EndIf

\State{$G_\Colors'' \gets$ The result of applying \cref{alg:blowUp} to every multi-edge in $G'$.}\label{alg:FPT:blowUp}

\State \Return $\mu(G_\Colors'')\leq (\dimension'-1+\constB)+2M$ \Comment{using FPT algorithm of \Cref{thm:blackbox}.}

\end{algorithmic}
\end{algorithm}

\begin{proof}[Proof of \Cref{thm:ourFPT}]
We prove that \Cref{alg:FPT} is correct and runs in time $2^{f(\constB)}\cdot \dimension^4$. 
Correctness follows from \cref{lem:G'}, \cref{lem:d+l-separable=>AtMostk2Multiedges} and \cref{lem:blowUp}.
Clearly, all steps except the invocation of the FPT algorithm of \Cref{thm:blackbox} in the last line can be performed in $O(\dimension^2+\constB^2)$.

We choose \const such that when we call the FPT algorithm of \Cref{thm:blackbox} with $G_\Colors''$ and \const it decides $\mu(G_\Colors'')\leq (\dimension' -1 + \constB)+2M$, i.e., we
choose \const such that $(\dimension' - 1 +\constB)+2M= \bound(G_\Colors'')+\const$.
Therefore let $\const:=((\dimension'-1+\constB)+2M) - \bound(G_\Colors'')$.
For bounding the runtime of this invocation, we need to check that $\const$ is dependent only on $\constB$.
Recall that by \Cref{obs:boundedErdosBound} we can bound $(\dimension'-1+\constB)-\bound(G_\Colors')\leq \frac{5}{4}\constB+\frac{1}{4}$, a quantity depending only on our parameter $\constB$. Furthermore, applying \Cref{alg:blowUp} to a multi-edge of multiplicity $m$ changes $\bound$ by $\frac{3}{2}m$. Thus we have $\bound(G_\Colors'')=\bound(G_\Colors')+\frac{3}{2}M$. Since $M\leq\constB^2$, we get that $k$ is bounded by $O(\constB^2)$.
Thus, the final invocation of the algorithm of \Cref{thm:blackbox} runs in time $2^{O(\constB^2)}\cdot n^4$. 
\end{proof}

\section{Conclusion and Further Directions}

In conclusion, we proved that
\TwoThiefNecklaceSplitting on \dimension-separable necklaces has a unique solution and can be solved in polynomial time. Also \dimension-separability can be tested in polynomial time.
Furthermore, we showed that \TwoThiefNecklaceSplitting, which in general is known to be \PPA-complete, admits an FPT algorithm for the parameter \constB such that the input necklace is $(\dimension-1+\constB)$-separable. Lastly, we showed that testing $(\dimension-1+\constB)$-separability is also FPT, even though testing well-separation of point sets in $\Reals^\dimension$ is \coNP-complete.

The condition of $\dimension$-separability is only sufficient for uniqueness of the solution to \TwoThiefNecklaceSplitting. To the best of our knowledge, there is currently no known necessary condition for uniqueness.

As our main open question we wonder how our algorithm for \TwoThiefNecklaceSplitting can be extended to more general settings. Firstly, can we also find polynomial time algorithms for $k$-Thief-Necklace-Splitting under the constraint of $n$-separability? Secondly, instead of halving every color class, can we maybe find an algorithm to find any $(\alpha_1,\ldots,\alpha_n)$-cut? The existence of these cuts is also guaranteed by \Cref{lem:alphaHS}, however our algorithm really only works for halving, since if we are not halving, the solution is not guaranteed to split a color with two components in the bigger component. 

Another interesting followup question is whether one can lift the definition of \sep-separability into higher dimensions. In other words, for \dimension point sets $P = \{P_1, \dots, P_\dimension\}$ in $\Reals^d$, can each subset $A$ of $P$ be separated from $P\setminus A$ by \sep hyperplanes?
Well-separation then becomes $1$-separability. 
Thus, deciding \sep-separability for $k$ as input or even for the case $k=1$ is \coNP-hard. It is likely that special cases such as  $d$-separability or \dimension-separability are also hard to decide. While well-separation is also contained in \coNP, this is not clear for $k$-separability for $k>1$.
Like \TwoThiefNecklaceSplitting, which has a unique solution under the condition of \dimension-separability, one could also investigate whether there are other geometric problems which gain interesting properties under the condition of the input being well-separated, or $k$-separable for some $k$.

Finally, can we extend our FPT algorithm for deciding $\mu(G)\leq \dimension-1+\constB$ on semi-Eulerian multigraphs to work on all connected multigraphs? Furthermore, can we maybe also decide $\mu(G)\leq \bound(G)+\constB$ (to get a direct analogue of the algorithm of Crowston, Jones, and Mnich for multigraphs) and not just $\mu(G)\leq n-1+\constB$?

\clearpage
\bibliography{literature}

\newpage
\appendix
\section{Non-Odd Number of Points}\label{app:non-odd}
If we drop the assumption that every color of a necklace is a discrete set of points of odd cardinality, we need to slightly adjust our definitions and results. In this more general setting, we consider a color to be a union of finitely many intervals, or finitely many points.

\subsection{Unique Solutions}
In this section, we discuss the necessary changes for recovering \Cref{thm:uniquenessOfSolution}, the uniqueness of solutions to \TwoThiefNecklaceSplitting on $n$-separable necklaces with $n$ colors.

Recall the proof of \Cref{thm:uniquenessOfSolution}: We apply the $\alpha$-Ham-Sandwich theorem (\Cref{lem:alphaHS}) to get a unique halving cut for the necklace lifted to the moment curve. It is an essential part of \Cref{lem:alphaHS} that solution hyperplanes must contain a point from each point set. In the case where every color consists of an odd number of points, this was automatically guaranteed by our definition, since the only way to split an odd number of points into two equal parts is to place a split point on one point of each color. To ensure this in the more general setting, we adjust the definition of \TwoThiefNecklaceSplitting as follows:
\begin{enumerate}
\item For the $n$ split points $q_1,\ldots,q_n$, we must have $q_i\in P_i$ for all $i\in [n]$. In other words, we must pick one point from each color as a split point. %
\item The intervals are not considered open but half-open: All split point belongs to \thiefA.
\end{enumerate}
The first change allows us to apply the $\alpha$-Ham-Sandwich theorem, while the second change is necessary to ensure that we can still halve colors consisting of an even number of points.
With these changes, \Cref{thm:uniquenessOfSolution} still holds.

\subsection{Solving \TwoThiefNecklaceSplitting}
It is crucial to \Cref{alg:solveNecklaceSplitting} that when we consider a color consisting of two components, we can identify one component for which we know the existence of a solution with a split point in this component. We can then remove the other component. If the two components have different size, we can achieve this by simply keeping the larger component; it must contain the split point. Under the assumption that every color consists of an odd number of points, it is guaranteed that the two components have different size, however this does not hold in general.

If we adjust the definition of \TwoThiefNecklaceSplitting as outlined above, since there is a unique solution, there is a fixed choice of component in which we have to place our split point in. However, this choice depends on the parity of the number of split points placed in between the components, which is unknown to the algorithm at this point. For the algorithm to work, we thus have to stick with the original definition of \TwoThiefNecklaceSplitting (\Cref{def:TwoThiefNecklaceSplitting}) and only adapt our algorithm.

We adapt the algorithm, by removing lines \ref{alg:solveNecklaceSplitting:median1} and \ref{alg:solveNecklaceSplitting:median2}, since it is no longer necessary to guarantee that the remaining component has odd cardinality. Furthermore, when we adjust the split point on line \ref{alg:solveNecklaceSplitting:adjustSplit}, we allow the algorithm to shift the split point beyond the end of the component (in which case the split point is placed between the end of the component and the beginning of the next color). With these adjustments, the algorithm still works.

Note that we were not able to find a version of \Cref{def:TwoThiefNecklaceSplitting} that allows for both \Cref{thm:uniquenessOfSolution} as well as \Cref{alg:solveNecklaceSplitting} to work at the same time.

\end{document}